\documentclass[12pt]{amsart}

\usepackage{amsfonts}
\usepackage{amssymb}
\usepackage{amsthm}
\usepackage{amsmath}
\usepackage{enumitem}
\usepackage{graphicx}
\usepackage{geometry}

\newtheorem{theorem}{Theorem}[section]
\newtheorem{lemma}[theorem]{Lemma}
\newtheorem{remark}[theorem]{Remark}

\newtheorem{proposition}[theorem]{Proposition}

\newtheorem{corollary}[theorem]{Corollary}
\newtheorem{conjecture}[theorem]{Conjecture}

\numberwithin{equation}{section}
\newcommand{\F}{{\mathcal{F}}}

\begin{document}
\title{On the existence of closed $C^{1,1}$ curves of constant curvature}

\author{Daniel Ketover}\address{Institute for Advanced study, 1 Einstein Drive Princeton, NJ 08540}
 \email{dketover@math.princeton.edu}
\author{Yevgeny Liokumovich}\address{Institute for Advanced study, 1 Einstein Drive Princeton, NJ 08540}
\email{ylio@ias.edu}
\maketitle\thanks

\begin{abstract}
We show that on any Riemannian surface for each $0<c<\infty$ there exists an immersed $C^{1,1}$ curve that is smooth and with curvature equal to $\pm c$ away from at most one point. We give examples 
showing that, in general, the regularity of the curve obtained by our procedure cannot be improved. \end{abstract}
\section{Introduction}
\noindent


A celebrated theorem of Birkhoff asserts that two-sphere endowed with a Riemannian metric contains a closed geodesic.  A long-standing question is the following conjecture of V.I. Arnold (\cite{A}, 1981-9): 


\begin{conjecture}[Arnold]
Every Riemannian two-sphere contains two smoothly immersed curves of curvature $c$ for any $0<c< \infty$. 
\end{conjecture}

Closed curves of curvature $c$ have a physical interpretation as closed orbits of a charged particle subject to a magnetic field.  The study of magnetic orbits has a long history and has been studied from many perspectives including Morse-Novikov theory, variational functionals, dynamical systems and Aubry-Mather theory. 

Two-spheres with arbitrary metric also contain at least three \emph{embedded} closed geodesics
(\cite{LS}, \cite{Ly}, \cite{G}).  
For convex two-spheres there is a min-max argument due to Calabi-Cao \cite{CC} to obtain at least one. 

The analogous conjecture for constant curvature due to S. Novikov (c.f. Section 5 in \cite{N}) asserts:

\begin{conjecture}[Novikov]
Every Riemannian two-sphere contains a smoothly embedded curve of curvature $c$ for any $0<c< \infty$. 
\end{conjecture}

The existence of immersed constant curvature $c$ curves is known on two-spheres with $c$ both large and small or when the two-sphere is nearly round (\cite{G1},\cite{G2}).  For $c$ very large, one can always obtain tiny geodesic circles of high curvature.  For $c$ near $0$, one can use the implicit function theorem to find constant curvature curves near a non-degenerate closed geodesic (obtained, say, from the min-max method of Birkhoff).  

On spheres of positive curvature it was shown by Schneider \cite{S1} that there exist at least two immersed constant curvature curves for each value of the curvature.  Thus the remaining issue in this case has been to obtain embeddedness.  The strongest result in this direction was obtained by Schneider \cite{S2} who showed that $1/4$-pinched surfaces always contain two embedded curves of any curvature $c$ (see also \cite{RS2} for the case $c$ small). 

We show on any surface (not necessarily a two-sphere) the existence of closed $C^{1,1}$ curves where the curvature may jump between $c$ and $-c$.  Our curves also are nearly \emph{nearly embedded}:

\begin{theorem}[Existence of $C^{1,1}$ curvature $\pm c$  curves] \label{main1}
Let $(\Sigma, g)$ be a closed Riemannian surface.  For each $0< c< \infty$ there exists a closed curve $\gamma_c$ such that one of the following holds:
\begin{enumerate}
\item $\gamma_c$ is smooth, quasi-embedded with curvature equal to $c$
\item There exists $p\in\Sigma$ such that in $\Sigma\setminus \{p\}$ the curve $\gamma_c$ is smooth, quasi-embedded, with curvature equal to $\pm c$.  Moreover, the tangent cone of $\gamma_c$ at $p$ consists of two distinct lines, and $\gamma_c$ is $C^{1,1}$ through $p$.
\end{enumerate}
\end{theorem}
\begin{remark}
Note that in case (2), the curve $\gamma_c$ is a kind of ``figure 8," but the curvature changes sign at the crossing point $p$ and is thus only $C^{1,1}$ through $p$.
\end{remark}
A curve is \emph{quasi-embedded} if it is either embedded or has finitely many self-touching points, i.e. points $p$ where the tangent cone consists of a line with multiplicity two, and in a small ball around the point the curve consists of two arcs meeting tangentially at $p$. 

We show in Section \ref{examples} that Theorem \ref{main1} is sharp in the sense that on flat tori our construction requires a $C^{1,1}$ point.  Similarly it is well-known (cf. Example 3.7 in \cite{G2}) that on closed hyperbolic surfaces there are no closed curvature $1$ curves (all horoballs have \emph{dense} projection to a closed surface). Our procedure produces examples with a $C^{1,1}$ point in these cases as well.  Thus the $C^{1,1}$ regularity is a threshold beyond which solutions may not exist.

On the three-legged starfish (cf. Figure 2 in \cite{CC}), which is known to have Almgren-Pitts width realized by a smooth curve $\alpha$ with an immersed point, our procedure for $c$ small produces a $C^{1,1}$ curve $\beta_c$ that converges smoothly away from the $C^{1,1}$ point to $\alpha$ as $c\rightarrow 0$. On the other hand, by the implicit function theorem if the metric is generic one should be able to produce a smoothly immersed curve $\gamma_c$ with curvature equal to $c$ converging smoothly to $\alpha$ as $c\rightarrow 0$.  Thus Theorem \ref{main1}, in some sense, picks out the non-smooth curve in this example.


Our proof of Theorem \ref{main1} uses methods of geometric analysis and a key ingredient is the recent min-max theory for the $\mathcal{F}_c$ functional due to Zhou-Zhu \cite{ZZ}, \cite{ZZ2}.  This functional is defined as follows.  Fix a closed Riemannian surface $\Sigma$ and a value $c>0$.  Given an open set $R\subset\Sigma$ with smooth boundary, define
\begin{equation}\label{cfunct}
\mathcal{F}_c(\partial R) = \mbox{Length}(\partial R) - c\mbox{Area}(R). 
\end{equation}

Critical points of this functional are curves of curvature $c$.  If one considers a one parameter family of open sets $\{R_t\}_{t\in [0,1]}$ sweeping out the manifold in the sense that they begin at $R_0=\emptyset$ and end at $R_1=\Sigma$, one sees that $\mathcal{F}_c(R_0)= 0$ and $\mathcal{F}_c(R_1)=-c\mbox{Area}(\Sigma)$.  By the Isoperimetric Inequality, the function $\mathcal{F}_c(R_t)$  has a positive derivative at $t=0$, and thus a critical point of the $\mathcal{F}_c$ functional must exist by general min-max principles.   In higher dimensions, Zhou-Zhu showed that one obtains a smooth quasi-embedded hypersurface $X_c$ of curvature equal to $c$.

In two dimensions, Zhou-Zhu showed \cite{ZZ2} (as in Pitts' \cite{P} work) that their method gives rise to a geodesic net with curvature equal to $\pm c$.  The goal of this paper is to upgrade the regularity from a constant curvature net to a $C^{1,1}$ immersion.

One advantage of using geometric measure theory applied to the $\mathcal{F}_c$ functional is that we only consider it on open sets in the surface, and thus the length and area terms in \eqref{cfunct} are always bounded.  If one considers the functional \eqref{cfunct} on \emph{immersions}, one has to rule out a sequence of curves whose $\mathcal{F}_c$ functional is converging to the critical value but whose areas and lengths are diverging.


Note that Theorem \ref{main2} does \emph{not} assert that the curvature $\pm c$ curves arise as the width for the $\mathcal{F}_c$ functional.   However, with more work and a suitable deformation theorem one can also obtain this additional information.

\begin{remark}
We have been informed that C. Bellettini, O.Chodosh,  C. Mantoulidis have independently obtained a result similar to Theorem \ref{main1} using the Allen-Cahn approach to min-max constructions. 
\end{remark}

\begin{figure} 
   \centering	
	\includegraphics[scale=0.75]{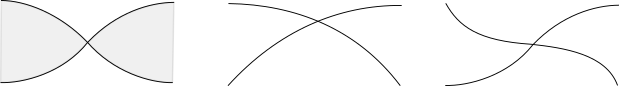}
	\caption{Min-max constant curvature
	curves can have a $C^{1,1}$ point
	like the curve on the left. (Note that
	it is $c$-convex to the shaded region).
	Intersections as in the
	second and the third diagram 
	cannot occur.}
	\label{fig:c11}
\end{figure}

\emph{Acknowledgements}: D.K. would like to thank Otis Chodosh, Rob Kusner and Yair Minsky for conversations. The authors wish to acknowledge the hospitality of the Institute for Advanced Study where part of this work was conducted and F.C. Marques for organizing the special year there on Variational Methods.  D.K. was partially supported by NSF grant DMS-1401996.  Y.L. was partially supported by NSF grant DMS-1711053.

\section{Existence of constant curvature network}
Let $\Sigma$ be a Riemannian surface and fix $c>0$.
For an open set $U \subset \Sigma$ of finite perimeter define
the functional $\F_c(\partial U)$:

$$\F_c(\partial U) = \mbox{Length}(\partial U) - c\mbox{Area} (U)$$

Consider the set $\Pi$ of all one parameter families $\{\Omega_t\}_{t=0}^1$ of open sets of $\Sigma$ with smooth boundary so that $\Omega_0=\emptyset$ and $\Omega_1=\Sigma$.  The `width' for the $\mathcal{F}_c$ functional is given by 
\begin{equation}
W_c=\inf_{\{\Omega_t\}\in\Pi}\sup_{t\in [0,1]}\mathcal{F}_c(\Omega_t).
\end{equation}

It follows from the isoperimetric inequality that $W_c>0$ for each $c>0$.

In \cite{ZZ} Zhou and Zhu developed a min-max theory 
for the functional $\F_c$ and produced a critical point of $\mathcal{F}_c$ at the critical value $W_c$.  In \cite{ZZ2} they address the one-dimensional case,
obtaining the existence of a network with constant curvature arcs, which is a stationary point for the functional $\F_c$.
Below we summarize their result and partly sketch the argument for the reader's convenience.

Let us first introduce some notation. Suppose $R\subset\Sigma$ is a region with boundary $\gamma$ consisting of finitely many smooth arcs that have curvature greater than or equal to $c$ relative to the region $R$, and which meet with (interior) angles between $0$ and $\pi$.  Let us call such a set \emph{$c$-convex to $R$}.  
Similarly, we call the boundary curve $\gamma$ \emph{$c$-concave to $R$}
if all arcs of $\gamma$ have curvature less than or equal to $c$ relative $R$ and
meet with (interior) angles between $\pi$ and $2 \pi$.

\begin{theorem} [Existence of constant curvature network (Zhou-Zhu)]\label{main2}
Let $(\Sigma, g)$ be a closed Riemannian surface.  For each $0 < c< \infty$ there exists a
set $R$ with $\partial R = \gamma$, such that

\begin{enumerate}
\item $W_c = \mathcal{F}_c(\partial R)$;

\item $\gamma$ is either smooth with curvature $c$ or there exists finitely many points $p_1,...p_k \in \Sigma$, such that $\gamma \setminus (p_1\cup p_2\cup...\cup p_k)$
is smooth and has constant curvature $\pm c$;

\item  The tangent cone to $\gamma$ at each $p_i$ consist of several half rays (with multiplicities) that sum to zero;

\item $\gamma$ is $c$-convex to $R$;


\end{enumerate}
\end{theorem}

\begin{proof}
It is shown in \cite{ZZ} that there exists a min-max sequence 
of boundaries $ \partial R_i$
converging to a varifold $V$ such that

\begin{itemize}
    \item $V$ has $c$-bounded first variation;
    
    \item For each $p\in\Sigma$, the set $V$ is  almost minimizing inside every proper annulus centered about $p$ of outer radius at most $r_{am}(p)>0$.
\end{itemize}

The fact that $V$ has $c$-bounded first variation implies that at each point $p$
any tangent cone at $p$ consists of half-rays (possibly with multiplicities) meeting at the origin summing to zero.

By a simple covering argument\footnote{Indeed, by compactness of $\Sigma$ the covering $\cup_{p\in\Sigma} B_{r_{am}}(p)$ has a finite subcover $\cup_{p\in\{p_1,...p_k\}} B_{r_{am}}(p)$ and in any small ball disjoint from $\{p_1,...p_k\}$ the almost minimizing property holds.} the almost minimizing property in annuli implies that, away from finitely many points $p_1,...p_k$, the varifold $V$ is almost minimizing in any small enough ball (where the radius of the ball depends on the center point $p$ and it may approach $0$ as $p$ converges to one of the points $p_i$). 

It is shown in \cite{ZZ2} that $V$
consists of smooth quasi-embedded arcs with multiplicity $1$ away from $p_1,...p_k$.  

Thus $V$ is a constant curvature network and is also a boundary of a region $R\subset\Sigma$. Since it is a critical point of the $\mathcal{F}_c$ functional, $V=\partial R$ is $c$-convex to $R$.

\end{proof}

\section{Maximum Principle and Constrained minimization}
Let us recall the maximum principle we will need.
Given a parameterized piecewise smooth curve $\gamma$ 
bounding a region $R$ let
$k_-(\gamma,s) = \lim_{\tau \rightarrow s^-} k(\gamma,\tau)$,
$k_-(\gamma,s) = \lim_{\tau \rightarrow s^+} k(\gamma,\tau)$
and let $\theta(\gamma,s)$ denote the interior angle (with respect 
to $R$) between the limiting tangent 
on the left and on the right at $\gamma(s)$.

\begin{lemma}[Maximum Principle]\label{maxprinciple}
Let $U_1$ and $U_2$ be two regions with
piecewise smooth boundaries $\gamma_1 = \partial U_1$
and $\gamma_2 = \partial U_2$. Suppose
$U_1 \subset U_2$ and $\gamma_1(s) = \gamma_2(s) \in \partial U_1 \cap \partial U_2$.
Then 

1. $\theta(\gamma_1,s) \leq \theta(\gamma_2,s)$;

2. if $\theta(\gamma_1,s) = \theta(\gamma_2,s)$ then
$k_-(\gamma_1,s) \geq k_-(\gamma_2,s)$ and
$k_+(\gamma_1,s) \geq k_+(\gamma_2,s)$.
\end{lemma}

Let us also record the first variation formula for $\mathcal{F}_c$. If $\gamma(s)$ is closed embedded curve bounding a region $R$ with normal $n(s)$ (pointing inward to $R$) and $\phi(s)$ is a function along $\gamma(s)$, we have

\begin{equation}\label{firstvar}
\frac{d}{dt}\bigg |_{t=0}\mathcal{F}_c(\gamma+tn\phi)=-\int_\gamma \phi(s)(k(s)-c)ds.
\end{equation}

To extend sweep-outs we will also need: 

\begin{lemma}[Constrained Minimization]\label{constrained}
Let $\Sigma$ be a closed Riemannian surface and 
$R \subset \Sigma$ be an open set with piecewise smooth boundary. 

I. Suppose
$\partial R$ is $c$-convex to $R$ and has at least one point of curvature greater than $c$. Then there exists a family of sets $\{ P_t \subset R \}$,
such that $\F_c(\partial P_t) \leq \F_c(\partial R)$ and 
$P_1$ satisfies the following:

1. $P_1 = \emptyset$ or $\partial P_1$ is an embedded closed curve with 
constant curvature $c$;

2. $\partial P_1$ is contained in the interior of $R$;

3. $\partial P_1$ is a local minimum of $\F_c$.

II. Suppose
$\partial R$ is $c$-concave to $R$ and has at least one point of curvature less than $c$. Then there exists a family of sets $\{ U_t \supset R \}$,
such that $\F_c(\partial U_t) \leq \F_c(\partial R)$ and 
$U_1$ satisfies the following:

1. $U_1 = \Sigma$ or $\partial U_1$ is an embedded closed curve with 
constant curvature $c$;

2. $\partial U_1$ is contained in the interior of $\Sigma \setminus R$;

2. $\partial U_1$ is a local minimum of $\F_c$.
\end{lemma}

\begin{proof}
We prove I. The proof of II is analogous.

Since there exists a point $ p \in \partial R$ with curvature
$k(p)>c$ it follows from the first variation formula \eqref{firstvar} that we can deform an arc of $\partial R$ near $p$
to obtain a new set $R' \subset R$ with $\F_c(\partial R') < \F_c(\partial R) - \epsilon$ for some small $\epsilon>0$.  Moreover, at all times along the deformation the $\mathcal{F}_c$ functional stays below $\mathcal{F}_c(\partial R)$.

Let $\mathcal{U}$ denote the set of open subsets 
of bounded perimeter $P \subset R$, such that 
there exists a homotopy $P_t$ from $R'$ to $P$
with $\F_c(\partial P_t) \leq \F_c(\partial R') + \epsilon/2$.
Note that the set $\mathcal{U}$ is closed and non-empty.

Let $P_1 \in \mathcal{U}$ denote a set with $\F_c(P_1) = \inf_{P \in \mathcal{U}} \F_c(P)$.
By the regularity theory for minimizers of $\F_c$ (Theorem 2.14 and Proposition 5.8 in \cite{ZZ})
the curve $\partial P_1$ is piecewise smooth and has constant curvature $c$ in the interior of $R$.  
If $\partial P_1$ had a self-intersection or a self-touching point $x$ in the interior
of $R$, then since rounding off the corners at $x$ increases area and decreases length, we can locally desingularize $\partial P_1$ to deform it into a 
smooth curve $\partial P'$ with $P' \subset R$
and $\F_c(\partial P')< \F_c (\partial P_1) $.
Hence, all arcs of $\partial P_1$ in the interior of $R$ must be smooth and embedded.

Suppose $\partial P_1 \cap \partial R \neq \emptyset$.
Note that we can not have $P_1 = R$ since 
there exists a point $ p \in \partial R$ with curvature
$k(p)>c$.
Let $a(s)$ be an arc of $\partial P_1$ whose interior is
contained in the interior of $R$ and with an endpoint $a(0) \in \partial R$.
By the Maximum Principle Lemma \ref{maxprinciple} 
we have that either $a$ meets $\partial R$ at an angle or
$k_+(0) \geq c$.
If the intersection is not tangential or $k_+>c$ then we can decrease the functional $\F_c$ by deforming $\partial P_1$.
If $k_+(0) = c$ then either $a(t)$ coincides with $\partial R$ for some
$t \in (0,\epsilon)$ or we have $k(t)>c$ for some $t \in (0,\epsilon)$.
Both possibilities contradict our choice of $a$.
Hence, $\partial P_1$ must lie in the interior of $R$.   
\end{proof}

\section{Rounding off Corners}
In this section we verify that we can round off corners of a $c$-convex set with a corner while both decreasing the functional $\mathcal{F}_c$ and increasing the curvature.  Rounding off corners certainly brings down the length, but it also decreases the area, and thus one has to look more carefully at its effect on the $\mathcal{F}_c$ functional.  But since the change in area is of smaller order than that in length, the $\mathcal{F}_c$ functional also goes down.  Indeed, after a blow-up argument the area term disappears and we are reduced to considering only the length functional.  

Alternatively, one can also apply curve shortening flow for short time as it will instantaneously smooth the curve and increase curvature (cf. Section \ref{sectioncflow}).

We need the following lemma used in the blow-up argument.  If $B_r(p)$ is a ball in a Riemannian surface with $r$ less than injectivity radius, define the dilation map 
\begin{equation}
d_r: B_r(x)\rightarrow \mathcal{B}_1\subset \mathbb{R}^2
\end{equation}
to be $d_r(y) = (\exp^{-1}_x(y))/r$.   For each $c>0$ let $\tilde{\mathcal{F}}_{c}$ denote the $\mathcal{F}_{c}$ functional on $\mathcal{B}_1$. 
\begin{lemma}[Rescaling of $\mathcal{F}_c$]\label{rescaling}
For $r$ small enough, there exists $\alpha(r)>0$ so that for all relative cycles $\Gamma$ in $B_r(p)$ bounding a region in $B_r(p)$ there holds:
\begin{equation}\label{rescaled}
(1-\alpha(r))\mathcal{F}_c(\Gamma)\leq  r \tilde{\mathcal{F}}_{rc}(d_r(\Gamma)) \leq   (1+\alpha(r)) \mathcal{F}_c(\Gamma), 
\end{equation}
where $\alpha(r)\rightarrow 0$ as $r\rightarrow 0$.  
\end{lemma}
\begin{proof}
Letting $\tilde{L}$ and $\tilde{A}$ denote the length and area on $\mathcal{B}_1(0)$ we have from the definitions 
\begin{equation}\label{resc}
r\tilde{\mathcal{F}}_{cr}(d(\Gamma))= r\tilde{L}(d(\Gamma))-cr^2\tilde{A}(d(\Gamma)).
\end{equation}
Letting $L$ denote length, and $A$ area on $B_r(p)$ we have: 
\begin{equation}\label{resc2}
(1-\beta(r)) L(\Gamma)\leq r\tilde{L}(d(\Gamma))\leq  (1+\beta(r)) L(\Gamma)
\end{equation}
and 
\begin{equation}\label{resc3}
(1-\gamma(r)) A(\Gamma)\leq r^2\tilde{A}(d(\Gamma))\leq  (1+\gamma(r)) A(\Gamma),
\end{equation}
where $\beta(r)$ and $\gamma(r)$ approach $0$ as $r\rightarrow 0$.  
Combining \eqref{resc}, \eqref{resc2} and \eqref{resc3} the claim follows. 
\end{proof}

\begin{proposition}[Rounding off Corners and Decreasing $\mathcal{F}_c$]\label{corners}
Let $B_r(p)$ be a ball in Riemannian surface $\Sigma$ of radius $r$ smaller than the injectivity radius of $\Sigma$.  Suppose two embedded curvature $c$ arcs $\{\alpha_1(t)\}_{t\in [0,1]}$ and $\{\alpha_2(t)\}_{t\in [0,1]}$ satisfy: 
\begin{enumerate}[label=(\roman*)]
\item $\alpha_1(1)=\alpha_2(1)=p$ and the two arcs meet at $p$ at angle $\theta<\pi$;
\item $\alpha_1\cap \alpha_2 = \{p\}$;
\item $\alpha_1(t)$ and $\alpha_2(t)$ are contained in the interior of $B_r(p)$ for $0<t\leq 1$;
\item  $\alpha_1(0)$ and $\alpha_2(0)$ are contained in $\partial B_p(r)$;
\item $\alpha_1\cup\alpha_2$ is $c$-convex to the acute wedge $W\subset B_r(p)$ between them
\end{enumerate} 

Let $\Gamma$ denote $\alpha_1\cup\alpha_2$.  Then for all $\epsilon>0$ small enough, there exists a curve $\Gamma_\epsilon$ with boundary in $\partial B_r(p)$ that is $c$-convex to a region $R_\epsilon$ so that 
\begin{enumerate}
\item $\Gamma_\epsilon$ coincides with $\alpha_1\cup\alpha_2$ at $\partial B_r(p)$;
\item $\Gamma_\epsilon\rightarrow \Gamma$ as $\epsilon\rightarrow 0$;
\item $\mathcal{F}_c(\Gamma_\epsilon)<\mathcal{F}_c(\Gamma)$;
\item $R_\epsilon\subset R_\delta$ whenever $\delta>\epsilon$;
\item $\Gamma_\epsilon$ is a $C^{1,1}$ curve that is smooth away from two points;
\item $\Gamma_\epsilon$ has curvature between $\frac{1}{2\epsilon}$ and $\frac{2}{\epsilon}$ at all points where it is smooth (in particular the curvature is greater than $c$).  
\end{enumerate}
\end{proposition}

\begin{proof}
For each $\epsilon>0$, consider the the $\epsilon$-tubular neighborhood $T_1$ about $\alpha_1$ and $T_2$ the $\epsilon$-neighborhood about $\alpha_2$.  When $\epsilon$ is small enough, we have that $\partial T_1\cap\partial T_2\cap W$ is singular at a unique point $p_\epsilon$.  The circle $C_\epsilon$ of radius $\epsilon$ about $p_\epsilon$ lies in $W$ and is tangent to $\alpha_1$ at $x_\epsilon$ and tangent to $\alpha_2$ at $y_\epsilon$ and has curvature $1/\epsilon+O(\epsilon).$  

We construct the curve $\Gamma_\epsilon$ by removing from $\alpha_1\cup\alpha_2$ the sub-arc beginning at $x_\epsilon$ and going to $p$, as well as the sub-arc beginning at $y_\epsilon$ and going to $p$ and adding in their place the segment of $C_\epsilon$ so that the resulting curve $\Gamma_\epsilon$ is $C^{1,1}$ over $x_\epsilon$ and $y_\epsilon$.  

It remains to verify (3).  Suppose toward a contradiction that there exists $\epsilon_i\rightarrow 0$ so that
\begin{equation}\label{bad}
\mathcal{F}_c(\Gamma_{\epsilon_i})\geq \mathcal{F}_c(\Gamma).
\end{equation}
Suppose without loss of generality that $\delta_i:=\mbox{dist}_\Sigma(p,x_{\epsilon_i})\geq \mbox{dist}_\Sigma(p,y_{\epsilon_i})$ for some subsequence of $i$.  
Note that there exists $c>0$ and $C>0$ so that 
\begin{equation}\label{control}
c\epsilon_i < \delta_i < C\epsilon_i
\end{equation}
and 
\begin{equation}\label{distbounded}
\frac{\mbox{dist}_\Sigma(p,x_{\epsilon_i})}{\mbox{dist}_\Sigma(p,y_{\epsilon_i})} \rightarrow 1.
\end{equation}

Let us consider the family of dilations $d_{\delta_i}$ about $p$ and the corresponding curves $\tilde{\Gamma}_i:=d_{1/\delta_i}(\Gamma_{\epsilon_i})$ as well as $\tilde{\Omega}_i:=d_{1/\delta_i}(\Gamma)$.  The curves $\tilde{\Gamma}_i$ and $\tilde{\Omega}_i$ are defined on the ball of radius $1$ $\mathcal{B}_1 \subset \mathbb{R}^2$.  

 The curves $\tilde{\Omega}_i$ have bounded curvature and thus converge to a set consisting of two rays $r_1$ and $r_2$ in $\mathcal{B}_1$ meeting at angle $\theta<\pi$ at the origin and extending to $\partial \mathcal{B}_1$.  Let us denote by $a$ the point $r_1\cap \partial \mathcal{B}_1$ and by $b$ the point $r_2\cap \partial \mathcal{B}_1$.
  
 The curves $\tilde{\Gamma}_i$ on the other hand consist of two pieces.  The circular part has curvature $\delta_i/\epsilon_i+O(\delta_i)$ and thus in light of \eqref{control} after passing to a subsequence they converge to an arc $A$ of a circle of positive curvature.  The arc $A$ contains the point $a$ by the choice of $\delta_i$ and because of \eqref{distbounded} also the point $b$.  In light of \eqref{distbounded} the other components of $\tilde{\Gamma}_i$ do not contribute to the limit. 
  
 
 Because of the Rescaling Lemma \ref{rescaling} and \eqref{bad} we obtain
 \begin{equation}\label{impossible}
 \mbox{Length}(A)\geq \mbox{Length}(r_1\cup r_2)=2.  
 \end{equation}  
 
 But \eqref{impossible} implies that a circular arc $A$ in $\mathbb{R}^2$ has larger length than the two line segments making a corner at $0$.  This is a contradiction and thus (3) holds.  
\end{proof}

\section{Proof of Main Theorem}
\begin{theorem}[Existence of $C^{1,1}$ curvature $\pm c$ curves] \label{main3}
Let $(\Sigma, g)$ be a closed 
Riemannian surface.  For each $0< c< \infty$ there exists a closed curve $\gamma_c$ such that one of the following holds:
\begin{enumerate}
\item $\gamma_c$ is smooth, quasi-embedded with curvature equal to $c$ 
\item There exists $p\in\Sigma$ such that in  $\Sigma\setminus\{p\}$ the curve $\gamma_c$ is smooth, quasi-embedded, with curvature equal to $\pm c$.  Moreover, the tangent cone of $\gamma_c$ at $p$ consists of two distinct lines and $\gamma_c$ is $C^{1,1}$ through $p$. 
\end{enumerate}
\end{theorem}
\begin{remark}
If we are in case (2) in the theorem, then the sign of the curvature changes as one passes through the $C^{1,1}$ point.  
\end{remark}

\begin{proof}
Without any loss of generality
we can assume that there are no smooth stable quasi-embedded curvature $c$ curves.  
Let $R$ be a region with $c$-convex boundary 
satisfying conclusions of 
 Theorem \ref{main2}.
 Let $\gamma = \partial R$.

If $R$ has smooth quasi-embedded (or embedded) boundary then we are done.
Hence, without loss of generality we may assume that there exist  points $p_1, ..., p_l$ with $m_1, ... , m_l$ arcs of $\gamma$ meeting at
each vertex for $m_i \geq 4$.
We claim that $l=1$ and $m_1=4$.

If not, we will construct a sweepout $\{ U_t \}$ of $\Sigma$
with $\F_c(\partial U_t)< \F_c(\partial R)$ for all $t$,
contradicting $\F_c(\partial R) = W_c$.

Consider a small ball $B$ centered at $p_1$ so that the set $\gamma\cap B$ consists of $m_1$ arcs $c_1,..., c_{m_1}$ meeting at $p_1$.   By the angle condition for stationary varifolds (item (3) in Theorem \ref{main2}) it follows that consecutive $c_i$ meet at angles strictly less than $\pi$.  For each $i=1,...m-1$ denote by $w^1_i$ the wedge in $B$ between $c_i$ and $c_{i+1}$ and denote by $w^1_{m_1}$ the wedge between $c_{m_1-1}$ and $c_1$.  Suppose without loss of generality that $w^1_1\cup w^1_3\cup...\cup w^1_{m_1-1}$ is equal to $B\cap R$. 

The construction of the sweepout $U_t$ proceeds as follows.
Let $\gamma_-^1$ denote the set obtained from $\gamma$
by replacing for each $i=1,3,...,m_1-1$ the arcs $c_i$ and $c_{i+1}$ of $\gamma$ near $p$ by a small arc
$a_i \subset R$ of curvature strictly greater than $c$.  One can do this provided the ball $B$ was chosen sufficiently small by Proposition \ref{corners}.\footnote{  Alternatively, one can apply the curve shortening flow for short time to make the curve smooth and with curvature strictly above $c$ at every point.}  In this way, the wedge $w^1_i$ has shrunk to a new set that we denote $w^1_i{}'$.   

Let $R_-^1 \subset R$ be the set bounded by $\gamma_-^1$.
Observe that when the arc is sufficiently small this 
deformation decreases the value of $\F_c$.  Denote by $G(w^1_i)=\mathcal{F}_c(\partial(R_-^1\cup w^1_i))-\mathcal{F}_c(\partial R_-^1)$.

Similarly, we modify $R_-^1$ in the neighbourhood
of $p_2$ to obtain region $R_-^2$.
Let $R_-$ denote the region we obtain after desingularizing
at every vertex. We have 
\begin{equation}\label{downgood}
\mathcal{F}_c(\partial R_-)=\mathcal{F}_c(\partial R)- \sum_{i \text{ odd},j} G(w_i ^j).
\end{equation}

Since we assumed that there are no stable embedded closed curves of
constant curvature $c$, 
by Lemma \ref{constrained} there exists a homotopy $R_t^-$
from $R_-$ to the empty set with $\F_c(\partial R_t^-) \leq \F_c (\partial R^-)$ for all $t$.

Similarly, we construct $R_+ \supset R$ by replacing for each $i=2,4,...,m_j$ the arcs $c_{i}, c_{i+1}$ (making a corner)
with arcs $b_i$ of curvature less than $0$ (curvature vector pointing outside of $R_+$).  The ``rounding off" of corners adds to area of $R$ and decreases the length of $\partial R$, and thus can be achieved with a homotopy strictly decreasing $\mathcal{F}_c$. Let us denote triangular region $(R_+\setminus R)\cap w_i^j$ by $t_i$. Let $H(w_i^j)=\mathcal{F}_c(\partial R)-\mathcal{F}_c(\partial (R\cup t_i))$.  

By Lemma \ref{constrained} there then exists a homotopy $R_t^+$
from $R_+$ to the set $\Sigma$ so that $\F_c(\partial R_t^+) \leq \F_c (\partial R^+)$ for all $t$.

To finish the construction of the desired sweepout we need to deform
$R_-$ into $R_+$ while keeping the value of functional $\F_c$ below $\F_c(\partial R)$.
Such a deformation exists when
the number of vertices $l \geq 2$ or when
$m_1\geq 6$ by bringing each component of $\partial R^+$ one at a time all the way up to $\partial R$ and desingularizing each time \footnote{See for instance Figure 3 in Calabi-Cao \cite{CC}.}.  Let us give more details.  

First we describe the deformation when $l=1$ and $m_1 \geq 6$. 
For convenience we drop the superscript $1$.

There is a homotopy supported in $w_1\cup w_3$ that brings $w'_1$ and $w'_3$ to $w_1$ and $w_3$, respectively (using the reverse of a localized piece of the ``rounding off" homotopy that was used to construct $R_-$ from $R$).  Since $m_1>4$ this homotopy has the property that the resulting region $R'$ satisfies (using \eqref{downgood})
\begin{equation}\label{downok}
\mathcal{F}_c(\partial R')=\mathcal{F}_c(\partial R)-G(w_5)-...-G(w_{m-1})< \mathcal{F}_c(\partial R).
\end{equation}

The new region $R'$ now contains $w_1$ and $w_3$ and the arcs $c_2$ and $c_3$ which bound the wedge $w_2$ in $\Sigma\setminus R$.  We then desingularize the corner between $c_2$ and $c_3$ (using a localized portion of the homotopy that we used in the construction of $R^-$) to homotope the region in $w_2$ to contain precisely $t_2$.  Thus we obtain a new region $R''$ so that 
\begin{equation}\label{downok2}
\mathcal{F}_c(\partial R'')=\mathcal{F}_c(\partial R')-H(w_2) < \mathcal{F}_c(\partial R').
\end{equation}

Note that in the course of moving between $R$ and $R''$ the largest the $\mathcal{F}_c$ functional has gotten is at $\partial R'$ satisfying \eqref{downok}.

Starting at $R_-$ we obtained a new region $R_2 :=R''$.  We will generate a sequence of increasing regions $R_-=R_1\subset R_2 \subset R_3 \subset...\subset R_{m/2-1}\subset R_{m/2} =R_+$ by iterating this procedure.  We have defined the first step of the iteration homotoping $R_1$ to $R_2$.  In the $i$th step, we define a homotopy $R_i(t)$ beginning at $R_i$ and ending at $R_{i+1}$ by first bringing $w'_{2i+1}$ to $w_{2i+1}$, and then adding in $t_{2i}$.  We thus have for $i=2,..,m/2-2$:

\begin{equation}\label{aa}
\mathcal{F}_c(\partial R_{i+1}) = \mathcal{F}_c(\partial R)-\sum^{m/2-2}_{j=i} G(w_{2j+3}) -\sum_{j=1}^i H(w_{2j}).
\end{equation}
while for all $0\leq t\leq 1$
we have 
\begin{equation}\label{bb}
\mathcal{F}_c(\partial R_i(t)) = \mathcal{F}_c(\partial R)-\sum_{j=i}^{m/2-2} G(w_{2j+3}) - \sum_{j=1}^{i-1} H(w_{2j}).
\end{equation}
\noindent
To construct $R_{m/2}$ from $R_{m/2-1}$ we add in the triangular region $t_m$.  Thus we have

\begin{equation}\label{aaa}
\mathcal{F}_c(\partial R_{m/2}) = \mathcal{F}_c(\partial R)- \sum_{j=1}^{m/2} H(w_{2j}),
\end{equation}
while for all $0\leq t\leq 1$
we have 
\begin{equation}\label{bbb}
\mathcal{F}_c(\partial R_{m/2-1}(t)) = \mathcal{F}_c(\partial R) - \sum_{j=1}^{m/2-1} H(w_{2j}).
\end{equation}

It follows from \eqref{downok}, \eqref{downok2}, \eqref{aa}, \eqref{bb}, \eqref{aaa}, and \eqref{bbb} that the $\mathcal{F}_c$ functional stays strictly below $\mathcal{F}_c(\partial R)$ in the homotopy from $R_-=R_1$ to $R_+=R_{m/2}$.  Since we have constructed a sweepout of $\Sigma$ so that the $\mathcal{F}_c$ functional on each slice is less than $\mathcal{F}_c(\partial R)$ we obtain a contradiction to the definition of width $W_c$.

Ruling out existence of more than one vertex is similar to the 
above argument. We first make a deformation in a small ball around
$p_1$, then around $p_2$ and so on. The value of $\F_c$ will
always be strictly smaller than $\F_c(\partial R)$ during the homotopy.

We conclude that  $\gamma$ has exactly one double point at $p_1$ with tangent cone consisting of two  lines.  If the lines coincide, then $\gamma$ is a quasi-embedded curve and $\gamma$ is smooth.  Suppose the lines are distinct.  The curve $\gamma$ cannot extend smoothly through $p$ as such a configuration of two curvature $\pm c$ arcs meeting transversally at $p$ is not stationary for the $\mathcal{F}_c$ functional. On the other hand $\gamma$ must extend as a $C^{1,1}$ immersion through $p$ as, along each arc passing through $p$, the normals coincide at $p$ but the curvature jumps from $c$ to $-c$, or vice versa.  

\end{proof}





\section{Examples}\label{examples}

\begin{figure} 
   \centering	
	\includegraphics[scale=0.75]{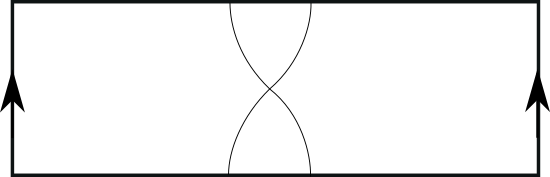}
	\caption{Min-max constant curvature curve
	on a thin torus has a $C^{1,1}$ self-intersection point.}
	\label{fig:torus}
	\end{figure}

In this section we show that Theorem \ref{main3} is sharp by exhibiting Riemannian surfaces whose width for the $\mathcal{F}_c$ functional is realized by a curve with a genuine $C^{1,1}$ singularity.

\subsection{Flat Tori}
Let $\mathcal{T}$ denote $\mathbb{R}^2/\mathcal{L}$ where $\mathcal{L}$ is the lattice generated by the unit vector in the $j$ direction and the vector of length $L>>1$ in the $i$ direction (i.e. an elongated rectangular torus with side lengths $1$ and $L$).  Denote by $\mathcal{F}$ the rectangular fundamental domain for $\mathcal{T}$ bounded by the four lines in $\mathbb{R}^2$ given by $x=0$, $y=0$, $x=L$ and $y=1$.  

Denote by $\mathcal{G}$ the closed geodesic:
\begin{equation}
\mathcal{G}:=\{(L/2,t)\subset\mathcal{T}\;|\; t\in [0,1]\}.
\end{equation}
For each $c>0$ there are precisely two circles $C_1$ and $C_2$ of radius $1/c$ in $\mathbb{R}^2$ passing through $(L/2,0)$ and $(L/2,1)$.   Consider the following subset of $\mathbb{R}^2$.
\begin{equation}
\mathcal{G}_c:=(C_1\cup C_2)\cap\mathcal{F}.
\end{equation}
Note that $\mathcal{G}_c$ can be parameterized as a $C^{1,1}$ curve in $\mathcal{T}$ that is smooth away from the point $(L/2,0)=(L/2,1)\subset\mathcal{T}$.  It can also be parameterized as the boundary of a lens-type disk.

We have the following well-known fact: 

\begin{lemma}
The Almgren-Pitts' width of $\mathcal{T}$ is realized by $2\mathcal{G}$.
\end{lemma}

For the c-widths, we show:

\begin{proposition}[c-Width of Flat Tori]
For $c$ close to $0$, the $c$-width of $\mathcal{T}$ is realized by $\mathcal{G}_c$.
Moreover, there holds for $c$ near $0$:
\begin{equation}\label{expansion}
\mathcal{F}_c(\mathcal{G}_c)= 2-\frac{1}{12}c^2+\mathcal{O}(c^3)
\end{equation}
and
\begin{equation}
\mathcal{G}_c\rightarrow 2\mathcal{G} \mbox{ in the sense of varifolds as } c\rightarrow 0.
\end{equation}
Finally for each $c\in (0,\infty)$ there exists an optimal sweepout $\{\Sigma_t\}_{t=-1}^1$ of $\mathcal{T}$ so that $\Sigma_{0}=\mathcal{G}_c$ and $\mathcal{F}_c(\Sigma_t)< \mathcal{F}_c(\mathcal{G}_c)$ whenever $t\neq 0$.   
\end{proposition}

\begin{remark}
Note that $2\mathcal{G}$ is a critical point of the $\mathcal{F}_c$ functional, and our interest in the expansion \eqref{expansion} is to see directly that it cannot arise as the $c$-width of the torus as $\mathcal{F}_{c}(\mathcal{G}_c) <\mathcal{F}_c(2\mathcal{G})=2$ and $\mathcal{G}_c$ can be exhibited in an optimal family. 
\end{remark}

\begin{proof}
Let $\gamma_c$ denote the $c$-width of $\mathcal{T}$.  Note that away from at most one point, since the curvature of $\gamma_c$ is bounded, we have that $\gamma_c$ converges to $2\mathcal{G}$ smoothly with multiplicity $2$\footnote{Note that \emph{any} pair of minimal length closed geodesics in $\mathcal{T}$ also realizes the Almgren-Pitts width, but only a geodesic with multiplicity $2$ can be realized as the limit of curvature $c$ curves that are smooth away from a point.}.  But curvature $c$ arcs are pieces of circles of radius $1/c$, and it is not hard to see that $\mathcal{G}_c$ is the only possible configuration that is smooth away from at most one point and converges to $2\mathcal{G}$ smoothly away from a point.


Let us verify the expansion \eqref{expansion}.  Straightforward trigonometry gives that the length $\mbox{Length}(\mathcal{G}_c)$ of $\mathcal{G}_c$ is:
\begin{equation}\label{length}
\mbox{Length}(\mathcal{G}_c)=\frac{4}{c}\arcsin(\frac{1}{2}c).
\end{equation}
The formula for the area of the lens-type disk bounded by $\mathcal{G}_c$ is given by
\begin{equation}
\mbox{Area}(\mathcal{G}_c)=2\arccos(\sqrt{1-c^2/4})-c^2\sqrt{\frac{1}{c^2}-\frac{1}{4}}.
\end{equation}
We want to compute:
\begin{equation}
\mathcal{F}_c(\mathcal{G}_c)=\mbox{Length}(\mathcal{G}_c)-c\mbox{Area}(\mathcal{G}_c),
\end{equation}
which is thus given as
\begin{equation}\label{dd}
\mathcal{F}_c(\mathcal{G}_c)=\frac{4}{c}\arcsin(\frac{1}{2}c))-\frac{2}{c}\arccos(\sqrt{1-\frac{c^2}{4}})+c\sqrt{\frac{1}{c^2}-\frac{1}{4}}
\end{equation}
Let us expand the first term in \eqref{dd} using the Taylor series: $\arcsin(x)=x+x^3/6+\mathcal{O}(x^5)$:
\begin{equation}\label{a}
\frac{4}{c}\arcsin(\frac{1}{2}c))= 2+\frac{1}{12}c^2+\mathcal{O}(c^4).
\end{equation}
We can expand the third term as
\begin{equation}\label{c}
c\sqrt{\frac{1}{c^2}-\frac{1}{4}}= 1-\frac{1}{8}c^2+\mathcal{O}(c^4).
\end{equation}
For the second term, we obtain the limit at $c=0$ using L'Hospital's rule:
\begin{equation}
\lim_{c\rightarrow 0} \frac{2}{c}\arccos(\sqrt{1-\frac{c^2}{4}})=\lim_{c\rightarrow 0} \frac{1}{\sqrt{1-\frac{c^2}{4}}}=1.
\end{equation}
Moreover, we can compute the derivative of $f(x) :=\frac{1}{\sqrt{x}}\arccos(\sqrt{1-\frac{x}{4}})$ at $0$ as
\begin{equation}
\dot{f}(0)= \lim_{x\rightarrow 0} \frac{1}{x}( \frac{1}{4\sqrt{1-x/4}}-\frac{1}{2}f(x)).
\end{equation}
By L'Hospital's rule again we obtain:
\begin{equation}
\dot{f}(0)=\frac{d}{dx}\bigg|_{x=0} \frac{1}{4\sqrt{1-x/4}}- \frac{1}{2}\dot{f}(0), 
\end{equation}
from which we obtain
\begin{equation}
\dot{f}(0)=\frac{1}{48}.
\end{equation}
Thus we get the Taylor expansion:
\begin{equation}\label{b}
\frac{2}{c}\arccos(\sqrt{1-\frac{c^2}{4}})=1+\frac{1}{24}c^2+\mathcal{O}(c^4).
\end{equation}
Plugging \eqref{a}, \eqref{b} and \eqref{c} back into \eqref{dd} we obtain \eqref{expansion}.
\end{proof}




\section{c-flow}\label{sectioncflow}

In this section, we introduce the $c$-flow, the gradient flow for the $\mathcal{F}_c$ functional. The results of this section are not needed in the paper but the flow is geometrically very natural and can give an alternative approach to some arguments in this paper. Namely, instead of rounding off corners of a $c$-convex region $R$, we can apply the $c$-flow for short time.   On the other hand, the $c$-flow seems badly behaved compared to the curve shortening flow  (for instance, it does not preserve embeddedness).  

If $\gamma$ is a closed curve bounding a disk $R$, let us denote $n_R$ the inward normal with respect to $R$ and let us define the curvature $k_\gamma$ of $\gamma$ with respect to this normal.   Thus a round circle the plane of radius $R$ bounding the disk has curvature $1/R$, while if we consider the circle bounding the complement of the disk, then it has curvature $-1/R$.  

Given a closed curve $\gamma_R$ bounding a region $R$, we consider the $\mathcal{F}_c$ functional
\begin{equation}
\mathcal{F}_c(\gamma_R)=\mbox{Length}(\gamma_R)-c\mbox{Area}(R).  
\end{equation}

The gradient flow for this functional is 
\begin{equation}\label{cflow}
\partial_t\gamma(x,t)=(k(x,t) - c)n_R(x,t).  
\end{equation}
so that we have
\begin{equation}
\frac{d}{dt}\mathcal{F}_c(\gamma(x,t)=-\int_{\gamma(x,t)} |k(x,t)-c|^2dx
\end{equation}

In $\mathbb{R}^2$ a round circle of radius greater than $1/c$ increases its radius under the $c$ flow without bound while circles of radii less than $1/c$ shrink to a point in finite time.  

This also shows that the $c$-flow does not preserve embeddedness.  Consider the $c$-flow on a flat torus. For $c$ very large, and an initial condition consisting of a circle of radius slightly greater than $1/c$, the circle under the $c$-flow will enlarge until it bumps into itself.

It can be useful to understand how the flow works on a round two sphere $\mathbb{S}^2$.  Consider the optimal foliation $\Sigma_t$ for $t\in [-1,1]$
\begin{equation}
\Sigma_t := \{(x,y,z)\; |\; x^2+y^2+z^2 =1 \mbox{ and } z=t\}
\end{equation} 
and corresponding regions 
\begin{equation}
R_t := \{(x,y,z)\;|\; x^2+y^2+z^2 =1 \mbox{ and } z\leq t\}. 
\end{equation}

There exists a smooth increasing function $f: (0,\infty) \rightarrow (0,1)$ so that the circle $\Sigma_{f(c)}$ has curvature $c$.  Only $\Sigma_{-f(c)}$ is stationary for the $\mathcal{F}_c$ functional even though $\Sigma_{f(c)}$ also has constant curvature equal to $c$.   For any tiny $\varepsilon>0$, If we apply the $c$ flow to $\Sigma_{-f(c)-\varepsilon}$ then it will foliate the region 
$R_{f(c)-\varepsilon}$ while if we apply the $c$ flow to $\Sigma_{-f(c)+\varepsilon}$ it will foliate the mean concave region $\mathbb{S}^2\setminus R_{-f(c)+\varepsilon}$ (and pass over the equator $\Sigma_0$).

One has the following evolution equation for the flow (cf. Lemma 1.3 in \cite{G})
\begin{lemma}\label{evolution}
\begin{equation}\label{evol}
k_t = k_{ss} + k^2(k-c) + G(k-c), 
\end{equation}
\noindent
where $G(x)$ denotes the Gaussian curvature of $\Sigma$, and $s$ is the arc-length parameter along $\gamma(x,t)$.    
\end{lemma}

Thus one has the following corollary of Lemma \ref{evolution} and the maximum principle:
\begin{corollary}\label{flow}
Let $\Sigma$ be a surface with $G(x)\geq 0$ and let $\gamma_0$ be a closed curve bounding a region $R$.  Let $\gamma_t$ be the evolution of the curve $\gamma$ by the $c$-flow \eqref{cflow}.    Then the condition 
\begin{equation}
k_{\gamma_t} \leq c
\end{equation}
holds whenever $k_{\gamma_0} \leq c$
and 
\begin{equation}
k_{\gamma_t} \geq c
\end{equation}
holds whenever $k_{\gamma_0} \geq c$.
\noindent
\end{corollary}

\begin{proof}
If $k\leq c$ initially then \eqref{evol} together with the fact that $G(x)\geq 0$ implies that the maximum of $k$ is  non-increasing in time.  Similarly if $k\geq c$ then  \eqref{evol} together with $G(x)\geq 0$ implies that the minimum of $k$ is non-decreasing in time.
\end{proof}
If $\gamma$ bounds a region $R$, let us say $\gamma$ is \emph{strictly $c$-convex} to $R$ if $\gamma$ is smooth and $k_\gamma>c$.  Similarly, we say $\gamma$ is \emph{strictly $c$-concave} to $R$ if it is smooth and  $k_\gamma <  c$. 

We have the following consequence of the strong maximum principle which alternatively could have been used in the proof of Lemma \ref{constrained}. 
\begin{lemma}\label{preserved} 
Suppose $\gamma_0$ is $c$-convex (resp. $c$-concave), has curvature somewhere equal to $c$ but not identically equal to $c$.  Then there exists a small time $t_0$ so that the flow 
\begin{equation} 
\partial_t\gamma(x,t)=(k(x,t) - c)n(x,t).  
\end{equation}
with initial condition $\gamma(x,t)=\gamma_0(x)$ is well-defined on $0<t<t_0$ and furthermore $\gamma_t$ is strictly $c$-convex (resp. strictly $c$-concave) for any $0 < t \leq t_0$.
\end{lemma}

\end{document}